\documentclass{article}
\usepackage{graphicx} 
\usepackage{mathtools}
\usepackage{amsfonts}
\usepackage{amsthm}
\usepackage{array}
\usepackage{soul}
\usepackage{amssymb}
\usepackage{graphicx}
\usepackage{float}
\usepackage{placeins}
\usepackage{bm}
\usepackage{amsmath}
\usepackage[dvipsnames]{xcolor}
\usepackage{braket}
\usepackage{tikz-cd}
\usepackage[colorlinks,linkcolor=blue]{hyperref}
\usepackage[nameinlink]{cleveref}
\usepackage{stmaryrd}
\usepackage{comment}

\newcommand{\R}{\mathbb{R}}
\newcommand{\Q}{\mathbb{Q}}
\newcommand{\Z}{\mathbb{Z}}
\newcommand{\I}{\mathcal{O}}
\newcommand{\C}{\mathbb{C}}
\newcommand{\q}{\left(\frac{a,b}{K}\right)}
\newcommand{\tens}[1]{%
\mathbin{\mathop{\otimes}\displaylimits_{#1}}%
}

\def\house#1{{%
    \setbox0=\hbox{$#1$}
    \vrule height \dimexpr\ht0+1.4pt width .4pt depth \dp0\relax
    \vrule height \dimexpr\ht0+1.4pt width \dimexpr\wd0+2pt depth \dimexpr-\ht0-1pt\relax
    \llap{$#1$\kern1pt}
    \vrule height \dimexpr\ht0+1.4pt width .4pt depth \dp0\relax
}}

\newtheorem{theorem}{Theorem}[subsection]
\newtheorem{lemma}[theorem]{Lemma}
\newtheorem{corollary}[theorem]{Corollary}

\newtheorem{proposition}[theorem]{Proposition}

\theoremstyle{definition}
\newtheorem{definition}[theorem]{Definition}
\newtheorem{example}[theorem]{Example}

\title{Well-Rounded Ideal Lattices from Totally Definite Quaternion Algebras}
\author{Yuan Xiang Chew\footnote{School of Physical and Mathematical Sciences, Nanyang Technological University, Singapore, {email:} S240055@e.ntu.edu.sg. The work of Yuan Xiang Chew was supported by a Nanyang President's  Graduate Scholarship, whose support is gratefully acknowledged.} and Fr\'ed\'erique Oggier }
\date{}

\begin{document}

\maketitle

\begin{abstract}
We study well-rounded ideal lattices from totally definite quaternion algebras. We prove existence and classification results, and illustrate our methods with examples.    
\end{abstract}


\section{Introduction}

Geometrically, a lattice $L$ in $\R^n$ is a discrete subgroup, whose points are described as integral linear combinations of $n$ linearly independent basis vectors, that is $L \coloneqq \set{\boldsymbol{x}M| \boldsymbol{x}\in \Z^n}$, where $M$ is a matrix called a generator matrix, which contains as rows basis vectors. 
The matrix $G\coloneqq MM^T$ is called a Gram matrix. 
%
%
The minimum $\|L\| $ of $L$ is defined by
$
\|L\| \coloneqq \operatorname{min}\set{\|\boldsymbol{x}\|^2| \boldsymbol{x} \in L\backslash{}\set{0}}.
$
All vectors of squared norm $\|L\|$ are called minimum vectors of $L$ and are collected in the set $S(L) \coloneqq \set{\boldsymbol{x} \in L| \|\boldsymbol{x}\|^2 = \|L\|}$. 
\begin{definition}\cite[p.14]{Martinet2010PerfectLI}
A lattice $L$ in $\R^n$ is said to be {\em well-rounded} if $span_\R(S(L)) = \R^n$, or equivalently, $S(L)$ contains an $\R$-basis for $\R^n$.    
\end{definition}

The property of being well-rounded generalizes that of being perfect (we recall that a perfect lattice $L$ in $\R^n$, as defined by Korkine and  Zolotarev, is a lattice completely determined by the set $S(L)$ in the sense that there is only one positive definite quadratic form taking value $1$ at all points of $S(L)$), since every perfect lattice is well-rounded. Many theoretical properties of well-rounded lattices are thus found in the literature devoted to the classification of perfect lattices (see  \cite{Martinet2010PerfectLI}).
Well-rounded lattices are also connected to the classical number-theoretic problem known as Minkowski's conjecture (see \cite{mcmullen2005minkowski}, where a two-step strategy is proposed to prove the conjecture, relying on well-rounded lattices). 

For applications of lattices in coding theory and cryptography, reduced bases such as the (Hermite–)Korkine–Zolotarev reduced basis are extensively studied (see e.g., \cite{porter2024new} for recent results). Well-rounded lattices are used in the context of lattice decoding using the Korkin-Zolotarev reduced basis \cite{banihashemi2002complexity}. 

Several works addressed the construction of well-rounded lattices from number fields (we recall that constructions of lattices from number fields, called ideal lattices, require a number field, its ring of integers, and an ideal, integral or fractional, of the ring of integers, equipped with a trace form \cite{Bayer-Fluckiger:1999}): \cite{fukshansky2012well} provides classification results, namely that a lattice built from a ring of integers is well-rounded  exactly when the number field is a cyclotomic field; \cite{alves2025well} provides a small correction regarding the hypotheses used in \cite{fukshansky2012well} which leads to the same conclusion. If a $\Z$-module or a fractional ideal that is not the ring of integers is used, the number field is not required to be a cyclotomic field, 
thus rendering meaningful the study of well-rounded lattices over different types of number fields such as cyclic cubic and quartic fields \cite{tran2023well}, and number fields of odd prime degree \cite{de2019well,bastos2025well}; the latter works consider the particular case of ramification.

In Section \ref{sec:quat}, we generalize the number field setting to quaternion algebras: number fields are replaced by totally definite quaternion algebras, rings of integers by orders, ideals of a ring of integers by ideals of orders. We also need a suitable trace form, and an embedding into $\R^n$. Section \ref{sec:WR} contains the main results, namely theorems discussing the existence and classification of well-rounded lattices coming from totally positive quaternion algebras. Examples and points of discussion are presented in Section \ref{sec:ex}.

\section{Ideal Lattices from Quaternion Algebras}
\label{sec:quat}

\subsection{Totally definite quaternion algebras}
\begin{definition}
A {\em quaternion algebra over a field $K$ (of  characteristic $\operatorname{char} K \neq 2)$} is a central simple algebra (a $K$-algebra whose center is $K$ and which has no nontrivial two-sided ideal) of dimension $4$ defined by
    \[A =\q =K \oplus{} Ki \oplus{}Kj \oplus{}Kij\] with a canonical $K$-basis $\set{1,i,j,ij}$ satisfying 
    \[i^2= a, \hspace{3mm} j^2 =b, \hspace{3mm} \text{and } ij=-ji\] where $a,b\in K^\times = K\backslash{} \{0\}$. 
\end{definition} 

The quaternion algebra $A=\left(\frac{-1,-1}{\R} \right)$ is known as Hamilton Quaternions, and is denoted by $\mathbb{H}$. 

Let $K$ be a number field of degree $n$ and signature $(s,t)$, i.e., $K$ has $s$ real embeddings and $t$ pairs of complex embeddings. We order the embedding set $Emb_\Q(K,\C) = \set{\sigma_m|m=1,\ldots,n}$ such that the first $s$ embeddings are real,  followed by $t$ pairs of complex conjugate embeddings.
Let $\I_K$ be the ring of integers of $K$. 



Set $A_\R \coloneqq A\tens{\Q}\R$ and let $M_2(\R)$, $M_2(\C)$ denote the ring of $2\times 2$ matrices with coefficients in $\R$, respectively $\C$. We have the following $\R$-algebra isomorphism \cite[Theorem 8.1.1]{MaclachlanReid}
 \begin{equation*}
     \begin{split}
         \varphi: A_\R  &\rightarrow \mathbb{H}^{s'} \times M_2(\R)^{s-s'} \times M_2(\C)^t \\
         x\tens{}r &\mapsto r(\sigma_1(x),\ldots,\sigma_n(x))
     \end{split}
 \end{equation*}
where $\sigma_m(x)$ is the action of $\sigma_m \in Emb_{\Q}(K,\C)$ on the coefficient of $x$ with respect to the canonical $K$-basis $\set{1,i,j,ij}$ and $s'\leq s$.

\begin{definition}\cite[Definition 14.5.7]{voight2021quaternion}
The quaternion algebra 
$A$ is said to be {\it totally definite} if $A_\R \cong \mathbb{H}^n$. In this case, $K$ is necessarily totally real, meaning that $s=n$ ($s'=n$ implies $s=n$).
 \end{definition}

For $x = x_0+x_1i +x_2j +x_3ij\in A$, there is a canonical involution on $A$ defined by $\overline{x}=x_0-x_1i -x_2j -x_3ij$. The reduced trace and reduced norm of $x$ are defined respectively as $tr_{A/K}(x) = x+\overline{x}$ and $n_{A/K}(x) = x\overline{x}$. They are elements of $K$.

We define a reduced trace map on $A_\R$ using the isomorphism $\varphi$ defined above.
For any $x\tens{}r \in A_\R$, set $\varphi(x\tens{}r)= (x_1,\ldots,x_n)$, then we have a map $tr_{A_\R/\R}: A_\R \rightarrow \R$ given by
\[ tr_{{A_{\R}}/\R}(x\tens{}r) = \sum_{i=1}^{s'}tr_{\mathbb{H}/\R}(x_i)+ \sum_{i=s'+1}^{s}tr_{M_2(\R)/\R}(x_i) + \sum_{i=s+1}^{s+t}tr_{M_2(\C)/\R}(x_i) \in \R.\]
Furthermore, the canonical involution on the quaternion algebra $A$ induces an $\R$-linear involution on $A_\R$ by $\overline{x\tens{}r} = \overline{x}\tens{}r.$

\subsection{Ideals and orders}

Let $A$ be a $K$-algebra where $K$ is a number field. The following definitions are from \cite{Reiner}.
\begin{definition}
An {\it ideal} (or $\I_K$-ideal) $I$ in $A$ is a finitely generated $\I_K$-module such that $I \tens{\I_K}K \cong A$. Equivalently, $I$ contains a $K$-basis for $A$.
\end{definition}
\begin{definition} 
An {\it order}  (or $\I_K$-order) $\Lambda$ of $A$ is an ideal of $A$ such that it is a subring of $A$ containing the multiplicative identity element $1$ from $A.$ An order is said to be {\it maximal} with respect to inclusion.
\end{definition}

For any ideal $I$, we can define the following two orders
\[ O_L(I)= \set{x\in A| xI \subseteq I},\hspace{5mm} O_R(I)= \set{x\in A | Ix\subseteq I}\] which are called the left order and right order of $I$ respectively.

A left ideal $I$ of $\Lambda$ satisfies $\Lambda I \subseteq I$ (i.e., $\Lambda \subseteq O_L(I)$), a right ideal $I$ of $\Lambda$ satisfies $I \Lambda \subseteq I$ (i.e., $\Lambda \subseteq O_R(I)$)
and a two-sided ideal $I$ of $\Lambda$ is stable under multiplication by $\Lambda$ on both sides. An ideal $I$ of $\Lambda$ is said to be {\it integral} if $I \subseteq \Lambda$. If $I$ is not integral, there exists $x \in \I_K$ such that $xI \subseteq \Lambda$ and $I$ is said to be fractional.

An equivalent definition of an ideal $I$ is a finitely generated torsion free $\I_K$-module \cite[p.44]{Reiner}. 

Since $I$ is finitely generated torsion free over $\I_K$, it is also finitely generated torsion free over $\Z \subseteq \I_K$ and we know that finitely generated torsion free $\Z$-modules are free of rank $4n$, thus all ideals (of $A$ or $\Lambda$) and orders have a $\Z$-basis.

The reduced norm of an ideal is defined as $n_{A/K}(I) =\set{n_{A/K}(x)|x \in I}\I_K$ which is a fractional ideal of $K$ \cite[Definition 16.3.1]{voight2021quaternion}. 
Let $A$ be a quaternion algebra over $K$, $\Lambda$ be an order of $A$ and denote its unit group by $\Lambda^\times.$

\begin{definition}
The {\it reduced norm one subgroup} is defined as 
\[
\Lambda^1 = \set{x\in \Lambda| n_{A/K}(x)=1}\subseteq \Lambda^\times
\]
which is also called the torsion subgroup of $\Lambda^\times.$    
\end{definition}

\subsection{Lattices}

We will work with the following algebraic definition of a lattice.

\begin{definition}
     A {\it lattice} is a pair $(L,b_\R)$ where $L$ is a free $\Z$-module of rank $n$ and $b_\R: L_\R \times L_\R \rightarrow \R$ is a positive definite symmetric bilinear form where $L_\R \coloneqq L \tens{\Z}\R$. 
\end{definition}

We first give a general definition of ideal lattices constructed from semi-simple algebras over $\Q$ (since we work with finite-dimensional algebras, semi-simple algebras are algebras that can be expressed as a Cartesian product of simple subalgebras).

\begin{definition}\cite[Definition 1.5.7]{Jerome2006}\label{Def: general ideal lattice}
Let $A$ be a semi-simple $\Q$-algebra of finite dimension and let $\Lambda$ be a $\Z$-order. Set $A_\R = A \tens{\Q}\R$ and let $\gamma$ be an $\R$-linear involution. An {\it ideal lattice} is the pair $(I,b)$ where $I$ is a right ideal of the order $\Lambda$ and $b: A_\R \times A_\R\rightarrow \R$ is a positive definite bilinear form satisfying 
\[ b(\boldsymbol{\lambda}\textbf{x},\textbf{y})= b(\textbf{x},\boldsymbol{\lambda}^\gamma \textbf{y}) \label{eqn: bilinear form condition}\] for any $\textbf{x},\textbf{y} \in I\tens{\Z} \R$ and $\boldsymbol{\lambda} \in \Lambda \tens{\Z}\R.$
\end{definition}

We note that an $\I_K$-order $\Lambda$ in a quaternion algebra $A$ over a number field $K$ is a $\Z$-order in a
semi-simple $\Q$-algebra, even though not all $\Z$-orders are $\I_K$-orders.

Under some restriction on $A$, for $A$ a quaternion algebra, and $K$ a number field, the symmetric bilinear form $tr_{A_\R/\R}$ defined in Subsection \ref{sec:quat} can be made positive definite and thus to satisfy the condition in Definition \ref{Def: general ideal lattice}. 

We recall that an element $x\in K^\times$ is called totally positive if $\sigma_m(x)>0$ for all $m=1,\ldots,n.$

\begin{proposition}\cite[Lemma 2.1, Proposition 2.2]{hou2017construction}\label{Prop: positive definite bilinear form}
Let $A$ be a totally definite quaternion algebra over the number field $K$, and $\alpha \in K^\times$ be totally positive. The map 
    \begin{equation*}
        \begin{split}
            b_\alpha: A_\R \times A_\R &\rightarrow  \R\\
            (\textbf{x,y}) &\mapsto tr_{A_\R/\R}(\alpha \textbf{x}\overline{\textbf{y}})
        \end{split}
    \end{equation*} is a positive definite symmetric bilinear form satisfying 
    \[ b_\alpha(\boldsymbol{\lambda}\textbf{x},\textbf{y})= b_\alpha(\textbf{x},\overline{\boldsymbol{\lambda}}\textbf{y})\]
    where $\alpha$ is understood as $\alpha \tens{}1$ 
\end{proposition}

From now on, we assume that $A= \q$ is a totally definite quaternion algebra (implying that $K$ is totally real) and that $\alpha$ is totally positive.
Since we will evaluate the bilinear form $b_\alpha$ with elements from $A$, we write $x=x\tens{}1$ for all $x\in A.$ The reduced trace can be rewritten as $tr_{{A_\R}/\R}(\alpha x\overline{y}) = Tr_{K/\Q}(tr_{A/K}(\alpha x\overline{y}))$ by a direct computation \cite[p.4]{hou2017construction}, where $Tr_{K/\Q}$ refers to the trace defined over $K$, that is $Tr_{K/\Q}(a)=\sum_{\sigma_m\in Emb_{\Q}(K,\C)} \sigma_m(a)$ for $a\in K$.

Next, we would like to construct an embedding of $A$ into $\R^{4n}$. Let $\iota: A \hookrightarrow A_\R$ be the inclusion map.  Note that since $A$ is totally definite, then $\sigma_m(a),\sigma_m(b)< 0 $ for all $\sigma_m \in Emb_\Q(K,\C)$. Let $\set{1,i',j',i'j'}$ be the standard $\R$-basis of $\mathbb{H}$ and set $\alpha_m = \sqrt{\sigma_m(\alpha)}$. Then, for $y=y_0+y_1i'+y_2j'+y_3i'j$, the map 
\begin{equation*}
    \begin{split}
        \psi_m : \mathbb{H} &\rightarrow \R^4\\
        y &\mapsto \sqrt{2}(\alpha_m y_0,\alpha_m\sqrt{-\sigma_m(a)} y_1,\alpha_m\sqrt{-\sigma_m(b)} y_2,\alpha_m\sqrt{\sigma_m(ab)} y_3)
    \end{split} 
\end{equation*}
is an $\R$-vector space isomorphism.

This induces the following $\R$-vector space isomorphism 
\begin{equation*}
    \begin{split}
        \psi: \mathbb{H}^n \rightarrow \R^{4n}\\
        (z_1,\ldots,z_n) &\mapsto (\psi_1(z_1),\ldots,\psi_n(z_n)).
     \end{split}
\end{equation*}
The composition $\psi \circ \varphi: A_\R \rightarrow \R^{4n}$ is an isometric isomorphism of $\R$-vector space between $A_\R$ equipped with $b_\alpha$ defined in Proposition \ref{Prop: positive definite bilinear form} and $\R^{4n}$ with 
the usual Euclidean dot product. To see this, we will show that both inner product coincide. Let $x\tens{}r \in A_\R$ where $x= x_0+x_1i+x_2j+x_3ij$, then 
\begin{equation*}
    \begin{split}
        b_\alpha(x\tens{}r,x\tens{}r) = tr_{A_\R/\R}( \alpha x\overline{x}\tens{}r^2) &= r^2\sum_{i=1}^n tr_{\mathbb{H}/\R}(\sigma_i(\alpha x\overline{x}))\\     
        &= r^2  \sum_{i=1}^n 2\sigma_i(\alpha x\overline{x})\\
        &= 2r^2 Tr_{K/\Q}(\alpha x\overline{x})\\
        &= r^2Tr_{K/\Q}(tr_{A/K}(\alpha x\overline{x})).
    \end{split}
\end{equation*}
On the other hand,
\begin{equation*}
    \begin{split}
        &~~~~\psi\circ \varphi (x\otimes{}r)\\
    &=\psi(r(\sigma_1(x),\ldots,\sigma_n(x)))\\
        &= r(\psi(\sigma_1(x)),\ldots, \psi(\sigma_n(x)))\\
        &=\sqrt{2}r(\alpha_1 \sigma_1(x_0), \alpha_1\sqrt{-\sigma_1(a)}\sigma_1(x_1),\alpha_1\sqrt{-\sigma_1(b)}\sigma_1(x_2),\alpha_1\sqrt{\sigma_1(ab)}\sigma_1(x_3)\\
        &\hspace{60mm} \vdots\\
        & \hspace{10mm}\alpha_n \sigma_n(x_0), \alpha_n\sqrt{-\sigma_n(a)}\sigma_n(x_1),\alpha_n\sqrt{-\sigma_n(b)}\sigma_n(x_2),\alpha_n\sqrt{\sigma_n(ab)}\sigma_n(x_3)).
    \end{split}
\end{equation*} Then the dot product of $\psi\circ \varphi(x\otimes{}r)$ with itself is  \[2r^2 \sum_{i=1}^n \sigma_i\left(\alpha (x^2_0 -ax^2_1 -bx^2_2 + ab x^2_3)\right)= r^2Tr_{K/\Q}(tr_{A/K}(\alpha x\overline{x})).\]

\begin{definition}\label{Def: canonical embedding}
    The composition $\sigma= \psi\circ \varphi \circ \iota$ which will be called the {\it canonical  embedding} of $\Q$-vector space is the map
    \begin{equation*}
    \begin{split}
        \sigma: A &\rightarrow \R^{4n}\\
         x 
        & \mapsto \sqrt{2}(\alpha_1\sigma_1(x_0), \alpha_1\sqrt{-\sigma_1(a)}\sigma_1(x_1), \alpha_1\sqrt{-\sigma_1(b)}\sigma_1(x_2),\alpha_1\sqrt{\sigma_1(ab)}\sigma_1(x_3), \\
        & \hspace{60mm} \vdots\\ & \hspace{7mm}\alpha_n\sigma_n(x_0), \alpha_n \sqrt{-\sigma_n(a)}\sigma_n(x_1), \alpha_n\sqrt{-\sigma_n(b)}\sigma_n(x_2),\alpha_n\sqrt{\sigma_n(ab)}\sigma_n(x_3))
    \end{split}
\end{equation*} where $x = x_0+x_1i+x_2j+x_3ij$ and $\alpha_i = \sqrt{\sigma_{i}(\alpha)}.$
\end{definition}

Let $I$ be any ideal and $\set{e_1,\ldots,e_{4n}}$ be a $\Z$-basis for $I$, then $\sigma(I)$ is a lattice in $\R^{4n}$ with $\Z$-basis $\set{\sigma(e_1),\ldots, \sigma(e_{4n})}$ and let $M$ be a generator matrix whose $i$-th row is the vector $\sigma(e_i)$. For any $x= x_0+x_1i+x_2j+x_3ij, y= y_0+y_1i+y_2j+y_3ij \in A$, then 
\begin{equation*}
    \begin{split}
        \sigma(x)\sigma(y)^T &= 2\sum_{m=1}^n \alpha^2_m\sigma_m(x_0y_0) - \alpha^2_m\sigma_m(ax_1y_1)- \alpha^2_m\sigma_m(bx_2y_2)+ \alpha^2_m\sigma_m(abx_3y_3)\\
        &=  \sum_{m=1}^n \sigma_m(2\alpha (x_0y_0 - ax_1y_1 -bx_2y_2 + abx_3y_3))\\
        &= \sum_{m=1}^n \sigma_m(tr_{A/K}(\alpha x\overline{y}))\\
        &= Tr_{K/\Q}(tr_{A/K}(\alpha x \overline{y})).
    \end{split}
\end{equation*}
The corresponding Gram matrix is $G = MM^T = \left(Tr_{K/\Q}(tr_{A/K}(\alpha e_i \overline{e_j})) \right)_{1\leq i,j\leq 4n}.$

\begin{definition}
    Let $A$ be a totally definite quaternion algebra over a number field $K$, $I$ be a right ideal of an $\I_K$-order $\Lambda$ and $\alpha\in K^\times$ be totally positive. An {\it ideal lattice} is then the pair $(I,b_\alpha)$ where $b_\alpha$ is the positive definite symmetric bilinear form defined in Proposition \ref{Prop: positive definite bilinear form}. We also write $\sigma(I)=(I,b_\alpha)$ where $\sigma$ is the canonical embedding.
\end{definition}

\section{Well-Rounded Lattices from Quaternion Algebras}
\label{sec:WR}

\subsection{Existence of well-rounded lattices}

Recall that 
$
\|(I,b_\alpha)\| \coloneqq \operatorname{min}\set{\|\boldsymbol{x}\|^2| \boldsymbol{x} \in (I,b_\alpha)\backslash{}\set{0}}.
$
We first give a lower bound on $\|(I,b_\alpha)\|$.
\begin{proposition}\label{Prop: Lower bound for Hermite invariant for quaternion algebra}
Let $A=\q$ be a totally definite quaternion algebra over the number field $K$ of degree $n$, $I$ be a right ideal of an $\I_K$-order $\Lambda$ and consider the ideal lattice $(I,b_\alpha)$ where $\alpha \in K^\times$ is totally positive. Then
    \[ \|(I,b_\alpha)\|\geq 2nN_{K/\Q}(\alpha)^\frac{1}{n}N_{K/\Q}(n_{A/K}(I))^\frac{1}{n}.\]  
\end{proposition}
\begin{proof}
    Let $0\neq x\in I$ be such that $\|\sigma(x)\|^2 = \|(I,b_\alpha)\|=Tr_{K/\Q}(tr_{A/K}(\alpha x\overline{x}))$. If we write $x = x_0 + x_1i + x_2j + x_3ij \in A$, then $tr_{A/K}(\alpha x\overline{x}) = tr_{A/K}(\alpha n_{A/K}(x))= 2\alpha (x^2_0- ax^2_1 -b x^2_2 +abx^2_3).$ Since $A$ is totally definite (in particular $K \subset \R$ and $\sigma_i(a),\sigma_i(b)<0$ for all $i$) and $\alpha$ is totally positive, then for each $i =1,\ldots, n$, 
    \begin{equation*}
        \begin{split}
            \sigma_i(tr_{A/K}(\alpha x\overline{x})) 
            &=\sigma_i( 2\alpha (x^2_0- ax^2_1 -b x^2_2 +abx^2_3))\\
            &=  2 \sigma_i(\alpha)( \sigma_i(x_0)^2-\sigma_i(a)\sigma_i(x_1)^2-\sigma_i(b)\sigma_i(x_2)^2 +\sigma_i(ab)\sigma_i(x_3)^2) \\&\geq 0.
        \end{split}
    \end{equation*} By the inequality of arithmetic and geometric means, we have that 
    \begin{equation*}
        \begin{split}
        \| \sigma(x)\|^2 
            &=Tr_{K/\Q}(tr_{A/K}(\alpha x\overline{x}))\\
            &= Tr_{K/\Q}(2\alpha n_{A/K}(x))\\
            &\geq n\left( N_{K/\Q}(2 \alpha n_{A/K}(x))\right)^\frac{1}{n}\\
            &=2n\ N_{K/\Q}(\alpha)^\frac{1}{n} N_{K/\Q}(n_{A/K}(x))^\frac{1}{n}.
        \end{split}
    \end{equation*}  
    
We are thus left to lower bound $N_{K/\Q}(n_{A/K}(x))$. 
Since $x \in I$, then $n_{A/K}(x) \in n_{A/K}(I)$ which is a fractional ideal in $K$. 

First, suppose that $I$ is integral, then $n_{A/K}(I) \subseteq \I_K$ since $n_{A/K}(\Lambda) = \I_K$. So $n_{A/K}(x)\I_K \subseteq n_{A/K}(I)\subseteq \I_K$ which implies that 
\begin{eqnarray*}
&& |\I_K/n_{A/K}(x)\I_K| \geq |\I_K/n_{A/K}(I)| \\
&\iff & N_{K/\Q}(n_{A/K}(x)\I_K) \geq N_{K/\Q}(n_{A/K}(I)). 
\end{eqnarray*}
    
Note that $N_{K/\Q}(n_{A/K}(x)\I_K) = |N_{K/\Q}(n_{A/K}(x))|= N_{K/\Q}(n_{A/K}(x))$ as $A$ is totally definite. Therefore, $N_{K/\Q}(n_{A/K}(x)) \geq N_{K/\Q}(n_{A/K}(I))$.

Suppose next that $I$ is not integral, then there exists $y \in \I_K$ such that $yI \subseteq \Lambda.$ Likewise, we have that $n_{A/K}(yI)\subseteq \I_K$ is an integral ideal which implies that $n_{A/K}(yx)\I_K \subseteq n_{A/K}(yI)\subseteq \I_K$ and thus 
\[ N_{K/\Q}(n_{A/K}(yx)) \geq N_{K/\Q}(n_{A/K}(yI)) \implies N_{K/\Q}(n_{A/K}(x))\geq N_{K/\Q}(n_{A/K}(I)).\]

\end{proof}
We denote by $\Q^\times_{>0}$ the set of positive rational numbers, which are totally positive in any number field $K$.

\begin{corollary}\label{Cor: shortest vector attained by reduced norm 1 elements}
Consider the ideal lattice $(\Lambda, b_\alpha)$ where $\Lambda$ is an $\I_K$-order of $A$ and $\alpha \in \Q^\times_{>0}$. 
Then
\[ 
\|(\Lambda,b_\alpha)\| = 2n\alpha,
\]
meaning that the lower bound for $\|(\Lambda,b_\alpha)\|$ given in Proposition \ref{Prop: Lower bound for Hermite invariant for quaternion algebra} is attained.
Furthermore, the lower bound is exactly attained by elements of $\Lambda^1$, the reduced norm one subgroup of $\Lambda$.  
\end{corollary}
\begin{proof}
   We note that since $\Lambda$ is an order of $A$, then $n_{A/K}(\Lambda) = \I_K$ and thus $N_{K/\Q}(n_{A/K}(\Lambda)) = 1.$ By assumption, $\alpha \in \Q^\times$ and thus from Proposition \ref{Prop: Lower bound for Hermite invariant for quaternion algebra}
    \[ \|(\Lambda,b_\alpha)\| \geq 2n \alpha.\] Consider $1 \in \Lambda$, then 
    \[\|\sigma(1)\|^2 = Tr_{K/\Q}(tr_{A/K}(\alpha))= Tr_{K/\Q}(2\alpha) = 2n \alpha,
    \] 
which implies that $\|(\Lambda,b_\alpha)\|= 2n\alpha$.

Furthermore, we will show that $x\in \Lambda$ attains the lower bound if and only if $x \in \Lambda^1.$ 

First, let $x \in \Lambda$ such that $\|\sigma(x)\|^2 = \|(\Lambda,b_\alpha)\|$, then 
    \[ 2n\alpha = \|\sigma(x)\|^2 = \alpha Tr_{K/\Q}( tr_{A/K}(x\overline{x})) \implies 2n = Tr_{K/\Q}(tr_{A/K}(x\overline{x})).\] We have shown in the proof of the previous proposition that for each $i=1,\ldots, n$,  $\sigma_i(tr_{A/K}(x\overline{x})) = \sigma_i(tr_{A/K}(n_{A/K}(x))) > 0$ and thus by the inequality of arithmetic and geometric means, we get
\begin{eqnarray*}
2n\alpha &=& Tr_{K/\Q}(tr_{A/K}(\alpha n_{A/K}(x))) \\
&\geq & n\alpha N_{K/\Q}(tr_{A/K}(n_{A/K}(x)))^\frac{1}{n}\\
&=&2n\alpha N_{K/\Q}(n_{A/K}(x))^\frac{1}{n}.    
\end{eqnarray*}
This implies that $1\geq N_{K/\Q}(n_{A/K}(x))$. We note that $N_{K/\Q}(n_{A/K}(x))$ is a positive integer since $x \in \Lambda$, this means that $1= N_{K/\Q}(n_{A/K}(x))$. This forces the above inequality to be an equality, and equality in the inequality of arithmetic and geometric means implies that all terms are the same, namely
$\sigma_i(tr_{A/K}(\alpha n_{A/K}(x)))= 2\alpha\sigma_i(n_{A/K}(x))$ are equal for all $i=1,\ldots,n$ i.e., $\sigma_i(n_{A/K}(x))$ are all equal. So
\[ 
1= N_{K/\Q}(n_{A/K}(x))= \prod_{i=1}^n\sigma_i(n_{A/K}(x)) = \sigma_1(n_{A/K}(x))^n \iff \sigma_1(n_{A/K}(x))=1
\] 
so $\sigma_i(n_{A/K}(x))=1$ for all $i=1,\ldots,n$ which implies that $n_{A/K}(x)= 1$. Therefore, $x\in\Lambda^1$ as desired.

Conversely, suppose that $x \in \Lambda^1$. It is enough to show that $Tr_{K/\Q}(tr_{A/K}(\alpha x\overline{x})) = 2n\alpha.$ Indeed
    \begin{equation*}
        \begin{split}
            Tr_{K/\Q}(tr_{A/K}(\alpha x\overline{x})) &= \alpha Tr_{K/\Q}(tr_{A/K}(n_{A/K}(x)))\\
            &= 2\alpha Tr_{K/\Q}(1)\\
            &= 2 n \alpha.
        \end{split}
    \end{equation*}
\end{proof}

The next example motivates the study of well-rounded ideal lattices from totally definite quaternion algebras by showing they exist.

\begin{example}\label{Example: Hurwitz order over the rationals}
    Let $A = \left( \frac{-1,-1}{\Q}\right)$ be a totally definite quaternion algebra over $\Q$ and consider the well-known Hurwitz order $\Lambda = \Z  + \Z i  + \Z j + \Z \left(\frac{1+i+j+ij}{2} \right)$ with $\Z$-basis $S= \set{1,i,j,\frac{1+i+j+ij}{2}}\subseteq \Lambda^1$. 
    
    Fix a totally positive element $\alpha \in \Q^\times_{>0}$, then $\sigma(S)$ contains the following vectors
    \begin{equation*}
        \begin{split}
            \sigma(1) &= \sqrt{2}(\sqrt{\alpha},0,0,0)\\
            \sigma(i) &=\sqrt{2}(0,\sqrt{\alpha},0 ,0) \\
            \sigma(j) &= \sqrt{2}(0,0,\sqrt{\alpha},0) \\
            \sigma\left(\frac{1+i+j+ij}{2}\right) &= \sqrt{2}\left(\frac{\sqrt{\alpha}}{2},\frac{\sqrt{\alpha}}{2},\frac{\sqrt{\alpha}}{2},\frac{\sqrt{\alpha}}{2}\right)
        \end{split}
    \end{equation*} which is clearly linearly independent over $\R$. The lattice points in $\sigma(S)$  all have a squared norm of $2\alpha$ which is the lower bound and thus $(\Lambda,b_\alpha)$ is well-rounded.
A generator matrix for this lattice is thus
\[
\sqrt{2\alpha}
\begin{bmatrix}
1 & 0 & 0 & 0 \\
0 & 1 & 0 & 0 \\
0 & 0 & 1 & 0 \\
1/2 & 1/2 & 1/2 & 1/2 \\
\end{bmatrix}.
\]
This is a generator matrix for the lattice $D_4$, scaled by a factor of $\sqrt{2\alpha}$.
\end{example}

\subsection{Characterization of well-rounded lattices}

The next result characterizes 
well-rounded lattices coming from an ideal lattice $(\Lambda,b_\alpha)$ where $\Lambda$ is an $\I_K$-order and $\alpha \in \Q^\times_{>0}$.

\begin{proposition}\label{Lemma: well rounded if and only if basis}
Let $A$ be a totally definite quaternion algebra over the number field $K$ of degree $n$, $\alpha\in \Q^\times_{>0}$ and $\Lambda$ be an $\I_K$-order of $A$. The ideal lattice $(\Lambda,b_\alpha)$ is well-rounded if and only if  $\Lambda^1$ contains a $\Q$-basis for $A$.
\end{proposition}
\begin{proof}

Suppose $(\Lambda,b_\alpha)$ is well-rounded. Recall that $\Lambda$ is a free $\Z$-module of rank $4n$. Then, by definition, there exists a subset $S= \set{x_1,\ldots,x_{4n}}$ such that $\sigma(S)=\set{\sigma(x_1),\ldots,\sigma(x_{4n})}$ is a set of minimal vectors which forms a $\R$-basis for $\R^{4n}$. By Corollary \ref{Cor: shortest vector attained by reduced norm 1 elements},  $S\subseteq \Lambda^1$. Since $A_\R \cong \R^{4n}$ as an $\R$-vector space under the isomorphism $\psi: \mathbb{H}^n\rightarrow \R^{4n}$, $\psi^{-1}(\sigma(S)) = \set{x_1\tens{}1,\ldots,x_{4n}\tens{}1}$ is an $\R$-basis for $A_\R$ and thus linearly independent over $\Q$ in $A_\R.$ This set spans $A\tens{} 1\subseteq A_\R$ where $A$ is of dimension $4n$ over $\Q$ and thus is a $\Q$-basis for $A \tens{} 1$ if and only if $S$ is a $\Q$-basis for $A.$ 
 \end{proof}

Since the existence of well-rounded lattices depends on understanding $\Lambda^1$ for $\Lambda$ an $\I_K$-order of $A$, we recall the following classification results.

\begin{lemma}[\cite{maire2006cancellation} \label{Lemma: classification of reduced norm 1 group}Classification of reduced norm 1 subgroup of an order]
 Let $\Lambda$ be an order of a totally definite quaternion algebra. Then $\Lambda^1$ is isomorphic to one of the following groups (the choice of $m$ will be discussed next):

\begin{tabular}{|c|c|} \hline
            Groups & Presentation  \\ \hline
            Cyclic of order $2m$& $\langle x| x^{2m} = 1 \rangle$\\ \hline
            Binary dihedral of order $4m$& $\langle x,y| y^{2m} =1, x^2=y^m, xyx^{-1}=y^{-1}\rangle$\\ \hline
            Binary tetrahedral of order 24 & $\langle x,y,z| x^2=y^3=z^3= xyz \rangle$\\ \hline
            Binary octahedral of order 48 & $\langle x,y,z|x^2=y^3=z^4 = xyz\rangle$\\ \hline
            Binary icosahedral of order 120 & $\langle x,y,z| x^2=y^3=z^5 =xyz\rangle$\\ \hline
       \end{tabular}
       \label{tab:placeholder}\\
       
The last three groups are also called exceptional groups and for each of these groups,  $(xyz)^2 = 1$ holds. 
\end{lemma}

The following theorem from $\cite{maire2006cancellation}$ classifies quaternion algebras and corresponding number fields.
\begin{theorem}[\cite{maire2006cancellation}]\label{Thm: Classification of groups}
    Let $A$ be a totally definite quaternion algebra over $K$. Then $A^\times$ contains a group isomorphic to 
    \begin{enumerate}
        \item a binary tetrahedral, octahedral or icosahedral group if and only if $K$ contains $\Q, \Q(\sqrt{2}),\Q(\sqrt{5})$ respectively and $A = \left(\frac{-1,-1}{K}\right)$;
        \item a binary dihedral group of order $4m$ where $m \geq 2$ if and only if $K$ contains $\Q(\zeta_{2m}+\zeta^{-1}_{2m})$ and $A = \left( \frac{-1,t^2-4}{K}\right)$ where $t = Tr_{K(\zeta_{2m})/K}(\zeta_{2m})$.
    \end{enumerate}
\end{theorem}

Equipped with the above classification results, we will next characterize when $\Lambda^1$ contains a $\Q$-basis for $A$ (corresponding to a well-rounded lattice).
To do so, we will use that $\Lambda^1$ contains a $\Q$-basis for $A$ if and only if $\Q\langle \Lambda^1\rangle = A$,
if and only if 
$\Q\langle S\rangle = A$
for $S\subseteq \Lambda^1$  a generating set for $\Lambda^1$ (since then $\Q\langle \Lambda^1\rangle = \Q\langle S\rangle$).
We will furthermore rely on the fact that since $A$ is totally definite, it is a division algebra over $K$,  which can be viewed as a division algebra over $\Q$, in which case it has dimension $4[K:\Q]$.  When considering a subfield $L$ of $A$, we make a distinction by denoting $L$ as a $K$-subfield if we view $A$ as a $K$-algebra and a $\Q$-subfield if we view $A$ as a $\Q$-algebra. 

Many of our arguments will rely on considering elements in $\Lambda^1$ for $\Lambda$ a $\I_K$-order, and look at the subfield these elements generate. It is thus useful to recall that given a field $F$:
\begin{enumerate}
\item 
For any $F$-algebra $A$ and any $a\in A$, 
we have a notion of minimal polynomial \cite[p.2]{Reiner} 
$m_{a,F}$ which is the unique monic polynomial in  $F[X]$ of least degree such that $a$ is a solution to this polynomial. If, furthermore, $A$ is a division algebra, then \cite[p.9]{drozd2012finite} 
$m_{a,F}$ is irreducible.
\item    
For any division $F$-algebra $A$ and $a \in A$, the $F$-subalgebra $\set{p(a)|p\in F[X]}$ is a subfield of $A$ \cite[p.9]{drozd2012finite}\label{Lemma: minimal polynomials are irreducible in division algebra} which we denote by $F(a)$, and can be viewed as a finite field extension of $F$.
\item For any division $\Q$-algebra $A$, an element $a\in A$ of order $m$ generates a subfield $\Q(a)$ such that $[\Q(a):\Q] = \varphi(m)$ where $\varphi$ is the Euler Totient function.
\end{enumerate}

We first prove that well-rounded lattices cannot come from a reduced norm one group that is cyclic.
In the lemmas below, $A$ is always a totally definite quaternion algebra over a number field $K$, and $\Lambda$ always an $\I_K$-order of $A$.   
\begin{lemma}\label{Lemma: cyclic groups cannot generate the whole algebra}
       If $\Lambda^1$ is cyclic, then $\Q\langle \Lambda^1 \rangle \neq A$.
   \end{lemma}
   \begin{proof}
Let $\Lambda^1 = \langle x \rangle$, 
so the $K$-algebra $K\langle \Lambda^1\rangle = K(x)$ is a $K$-subfield of $A$. 
Since any element $a \in A$ is a root of the polynomial $X^2 - tr_{A/K}(a)X + n_{A/K}(a) \in K[X]$, $[K(x):K]\leq 2$ which implies that $K(x)\neq A$. Therefore, $\Q\langle \Lambda^1\rangle \subseteq K\langle \Lambda^1\rangle \neq A$.

\end{proof}

When $\Lambda^1$ is isomorphic to the exceptional groups, the dimension of the $\Q$-algebra $\Q\langle \Lambda^1\rangle$ can be determined.

\begin{lemma}\label{Lemma: subfield can only contain one generator}
    Suppose $\Lambda^1$ is isomorphic to an exceptional group  stated in Lemma \ref{Lemma: classification of reduced norm 1 group}. Any $\Q$-subfield $L$ of $A$  contains at most one of its  generators. Furthermore, any $\Q$-subalgebra of $A$ containing two of its generators contains the entire group.
\end{lemma}
\begin{proof}
    The three exceptional groups all have three generators $x,y,z$, satisfying in particular $x^2=xyz$. Let $\Lambda^1$ be isomorphic to one of these three groups. Suppose that $L$ is a $\Q$-subfield that contains $x,y.$ Then
    \[ x^2 = xyz \iff  z = y^{-1}x \in L.\] So $x,y,z \in L$ which implies $\Lambda^1 \subset L$ which is a contradiction since $L$ is a field but $\Lambda^1$ is a non-abelian group.

    Let $B$ be a $\Q$-subalgebra of $A$ containing $x,y$. We know that $A$ is a $\Q$-division algebra since it is totally definite and any $\Q$-subalgebra of $A$ is also division. Then $z = y^{-1}x \in B$ and so, $\Lambda^1 \subset B.$
\end{proof}
\begin{proposition}
    \label{Lemma: subspace is infact an algebra}
   Suppose $\Lambda^1$ is isomorphic to an exceptional group stated in Lemma \ref{Lemma: classification of reduced norm 1 group}. The  $\Q$-vector space $B= \Q(z)\oplus{}\Q(z)y$ is a $\Q$-algebra. In particular, $\Q\langle \Lambda^1 \rangle = \Q\langle x,y,z\rangle = B. $
\end{proposition}
\begin{proof}
Let $x,y,z$, be the generators of the three exceptional groups. By Lemma \ref{Lemma: subfield can only contain one generator}, $y\not\in \Q(z)$, therefore $B = \Q(z)\oplus{}\Q(z)y$ is a $\Q$-vector subspace of $A.$

To obtain a $\Q$-algebra, we need to prove that $yz^k \in B$, $k \geq 1$, since we can then extend the multiplication  linearly to $y (\sum_{i=0}^k q_iz^i) = \sum_{i=0}^k q_iyz^i \in B$ for any $\sum_{i=0}^k q_iz^i \in \Q(z)$, so, for any $a_1+b_1y, a_2+b_2y \in B$, $a_1,a_2,b_1,b_2 \in \Q(z)$, we obtain 
    \[(a_1+b_1y)(a_2+b_2y)= a_1a_2+a_1b_2y + b_1(ya_2) + b_1(yb_2)y \in B.\]
   
We prove that $yz^k\in B$, $k \geq 1$, by induction on $k$.    
For the base case, we show that $yz \in B$. Since $x^2= xyz$ holds for the three exceptional groups, we get $x = yz$, so $x^2 = (yz)^2 = y(zy)z.$ Since $x^2=  y^3$ for the three exceptional groups, then $y^3 = y(zy)z \implies y^2 = zyz \implies z^{-1}y^2 = yz$. Since $y$ is of order $6$, its minimal polynomial over $\Q$ is $\Phi_6 = X^2-X+1$ and thus $y^2 = y- 1$. Therefore, 
    \[yz = z^{-1}y^2 = z^{-1}(y-1) =  - z^{-1} + z^{-1}y \in \Q(z)\oplus{}\Q(z)y\]
which proves the base case.

Suppose that $yz^k \in B$ for some $k > 1$. 
Then
    \[ yz^{k+1} = (yz^k)z = (a_1+b_1y)z = a_1z + b_1yz\] where $yz^k =a_1+b_1y \in B$ by induction hypothesis. The base case shows that $yz \in B,$ so $yz = a_2 +b_2y $ and thus,
    \[yz^{k+1}= a_1z + b_1(a_2+b_2y) = a_1z + b_1a_2+b_1b_2y \in B. \]
    This proves that $B$ is a $\Q$-algebra.

Lastly, we show that $\Q\langle x,y,z \rangle = B$. Clearly $y,z \in B$, then $x \in B$ by Lemma \ref{Lemma: subfield can only contain one generator} since $B$ is a division $\Q$-subalgebra of $A$. This proves that $\Q\langle x,y,z\rangle \subseteq B$. It is also clear that $\Q(z) \subseteq \Q\langle x,y,z\rangle$ and $y \in \Q\langle x,y,z\rangle$, so $B\subseteq \Q\langle x,y,z\rangle.$
 \end{proof}

\begin{corollary}
\label{Lemma: binary tetrahedral}
If $\Lambda^1$ is isomorphic to the binary tetrahedral group, then $\dim_\Q(A) \geq 4$ and $\dim_\Q (\Q\langle \Lambda^1\rangle) =4$.
\end{corollary}
\begin{proof}
By Theorem \ref{Thm: Classification of groups}, $K$ contains necessarily $\Q$ , so $\dim_\Q(A)= 4[K:\Q] \geq 4.$
  By Proposition \ref{Lemma: subspace is infact an algebra}, $\Q\langle \Lambda^1\rangle = B$. Then $dim_\Q(B) = 2[\Q(z):\Q]= 2\varphi(6) = 4.$
\end{proof}
\begin{corollary}
\label{Lemma: binary octa and icosahedral}
   If $\Lambda^1$ is isomorphic to the binary octahedral group or the  icosahedral group, then $dim_\Q(A) \geq 8$ and $dim_\Q (\Q\langle\Lambda^1 \rangle) =8.$
\end{corollary}
\begin{proof}
By Theorem \ref{Thm: Classification of groups}, $K$ contains a quadratic extension and thus, $dim_\Q(A) = 4[K:\Q] \geq 8.$
  By Proposition \ref{Lemma: subspace is infact an algebra},  $\Q\langle \Lambda^1\rangle = B$. Let $z$ be the generator of $\Lambda^1$ which is isomorphic to either the binary octahedral group or the icosahedral group of order $8$ and $10$ respectively. Then $dim_\Q(B)= 2[\Q(z):\Q]=  8$ in both cases since $\varphi(8) = \varphi(10) = 4$. 
\end{proof}

\begin{theorem}\label{Thm: sufficient condition for WR for exceptional group case}
Suppose $\Lambda^1$ is isomorphic to an exceptional group.
\begin{enumerate}
    \item  If $[K:\Q]=1$, $\Lambda^1$ contains a $\Q$-basis for $A$ if and only if $\Lambda^1$ is isomorphic to the binary tetrahedral group.
    \item If $[K:\Q]=2$,  $\Lambda^1$ contains a $\Q$-basis for $A$ if and only if $\Lambda^1$ is isomorphic to the binary octahedral group or the binary icosahedral group.
    \item If $[K:\Q]> 2$, $\Lambda^1$ does not contain a $\Q$-basis for $A$.
    \end{enumerate}
\end{theorem}
\begin{proof}
The dimension of $A$ over $\Q$ is $4n$ where $n = [K:\Q].$ We note that from Lemma \ref{Lemma: binary tetrahedral} and Lemma \ref{Lemma: binary octa and icosahedral}, $dim_\Q (\Q \langle \Lambda^1\rangle)= 4$ if $\Lambda^1$ is isomorphic to the binary tetrahedral group and $dim_\Q(\langle \Lambda^1\rangle) =8$ for the other two exceptional groups. 

For the case $K=\Q$, $\Lambda^1$ will only be isomorphic to the  binary tetrahedral group by Theorem \ref{Thm: Classification of groups}. So $dim_\Q (A) = 4 = dim_\Q(\langle \Lambda^1\rangle)$ and thus, $\Lambda^1$ contains a $\Q$-basis from $\Lambda^1$.

For the case where $K$ is quadratic, $dim_\Q (A) = 8$. So only the binary tetrahedral group case does not contain a $\Q$-basis for $A.$

For a general number field $K$ of degree greater than 2, $dim_\Q (A) \geq 12$. So $\Lambda^1$ being isomorphic to any of the exceptional group will fail to have a $\Q$-basis for $A.$
\end{proof}

We are left with  the case where $\Lambda^1$ is isomorphic to the binary dihedral group of order $4m$, whose presentation is
\[
\langle x, y |y^{2m}=1, x^2=y^m, xyx^{-1}=y^{-1}\rangle.
\]
By Theorem \ref{Thm: Classification of groups}, $K$ contains $\Q(y+y^{-1})$ which is of degree $\frac{\varphi(2m)}{2}$ over $\Q$. Since $\varphi(2m)$ divides $2[K:\Q]$, for a given $K$, there are only finitely many possible values of $m$. 


\begin{lemma}\label{Lemma: Q algebra generated is algebra}
    Suppose $\Lambda^1$ is isomorphic to the binary dihedral group of order $4m$ and consider the $\Q$-vector subspace $B = \Q(y)\oplus{}\Q(y)x$. Then $\Q\langle\Lambda^1 \rangle = B.$
\end{lemma}
\begin{proof}
    If $x \in \Q(y),$ then $\Lambda^1 \subseteq \Q(y)$ which is a contradiction since $\Lambda^1$ is non-abelian. So $B =\Q(y)\oplus{}\Q(y)x$ is a $\Q$-vector space. 

    It is well-known that every element in $\Lambda^1$ can be expressed as $y^kx^\ell$ where $k \in \set{0,\ldots, 2m-1}$ and $\ell \in \set{0,1}$. So \[\Q\langle \Lambda^1 \rangle = \Q\langle \set{y^kx^\ell| k \in \set{0,\ldots,2m-1}, \ell \in \set{0,1}}\rangle = \Q(y)\oplus{}\Q(y)x.\]
\end{proof}

Recall that a maximal subfield is a subfield $L/K$ that is not contained in any other proper subfield. Necessarily, $[L:K] \leq dim_K(A).$ In particular, $[L:K]=2$ for a maximal subfield in a quaternion algebra.

\begin{theorem}\label{Thm: characterization for WR binary dihedral case}
    Let $A$ be a totally definite quaternion algebra over a totally real number field $K$ and suppose that $\Lambda^1$ is isomorphic to the binary dihedral group of order $4m$ for $m>1$. Then $\Lambda^1$  contains a $\Q$-basis for $A$ if and only if  $K$ is a maximal real subfield of $\Q(\zeta_{2m})$ i.e., $\varphi(2m) = 2n$, $n=[K:\Q]$.
\end{theorem}
\begin{proof}
    Suppose $\Lambda^1$ contains a $\Q$-basis for $A$ and $K$ properly contains $\Q(y+y^{-1})$ so that $\varphi(2m) < 2n$. Let $B = \Q(y)\oplus{}\Q(y)x = \Q\langle \Lambda^1\rangle$ by Lemma \ref{Lemma: Q algebra generated is algebra}. However, $dim_\Q(B)= 2[\Q(y):\Q]= 2 \varphi(2m) < 4n$ which implies that $B \neq A$ and thus $\Lambda^1$ does not contain a $\Q$-basis for $A$ which is a contradiction.
    
   Conversely, suppose $K$ is the maximal real subfield of $\Q(y)$ which can be realized by $K = \Q(y + y^{-1})$. So $L = \Q(y) \supset \Q(y+y^{-1})=K$ is a maximal $K$-subfield of the $K$-algebra $A$. Since $x \notin L$ then the $K$-algebra $ L\langle x\rangle = L\oplus{}Lx = B$ 
   has dimension $dim_K(B) =dim_K (L) dim_L (B) = 4$ which implies that $A = B$. 
Thus, $\Lambda^1$ contains a $\Q$-basis from $\Lambda^1.$ 
\end{proof}

\begin{theorem}\label{Thm: Characterization of WR lattice}
Let $A$ be a totally definite quaternion algebra over a totally real number field $K$ of degree $n$. 
\begin{enumerate}
\item If $[K:\Q]=1$, $\Lambda^1$ contains a $\Q$-basis for $A$ if and only if $\Lambda^1$ is isomorphic to  the binary tetrahedral group or the binary dihedral group of order $4m$ where $\varphi(2m) = 2$.
\item If $[K:\Q]=2$, $\Lambda^1$ contains a $\Q$-basis for $A$ if and only if $\Lambda^1$ is isomorphic to the binary octahedral group, the icosahedral group or the binary dihedral group and  $K= \Q(\zeta_{2m}+\zeta^{-1}_{2m})$.
\item If $(K:\Q]>2$, $\Lambda^1$ contains a $\Q$-basis for $A$ if and only if $\Lambda^1$ is isomorphic to the binary dihedral group and $K= \Q(\zeta_{2m}+\zeta^{-1}_{2m})$.
    \end{enumerate}
\end{theorem}
\begin{proof}
This follows from  Theorem \ref{Thm: Classification of groups}, Theorem \ref{Thm: sufficient condition for WR for exceptional group case}, Theorem \ref{Thm: characterization for WR binary dihedral case} and Lemma \ref{Lemma: cyclic groups cannot generate the whole algebra}.

\end{proof}

\begin{example} 
We revisit Example \ref{Example: Hurwitz order over the rationals}, where we set $A= \left(\frac{-1,-1}{\Q}\right)$, let $\Lambda$ be the Hurwitz order and chose $\alpha \in \Q^\times_{>0}$. We directly showed that the corresponding lattice is well-rounded. 

Using Theorem \ref{Thm: Characterization of WR lattice}, it suffices to check the group isomorphism of $\Lambda^1$. By \cite[Theorem 11.5.14]{voight2021quaternion},  $\Lambda^1$ is isomorphic to the binary tetrahedral group of order 24 and thus $(\Lambda,b_\alpha)$ is well-rounded.
\end{example}

\begin{corollary}\label{corr: wellrounded lattice from order is sufficient to ideal lattice to be well rounded}
    Let $A$ be a totally definite quaternion algebra over a totally real number field $K$, $I$ be a right ideal of $\I_K$-order $\Lambda$ and $\alpha \in \Q^\times_{>0}$. If $\Lambda^1$ contains a $\Q$-basis for $A$, then $(I,b_\alpha)$ is a well-rounded lattice.
\end{corollary}
\begin{proof}
    Let $0 \neq \beta \in I$ such that 
    \[\|\sigma(\beta)\|^2 = Tr_{K/\Q}(tr_{A/K}(\alpha \beta\overline{\beta})) = \|(I,b_\alpha)\|.\] Suppose $\set{y_1,\ldots,y_{4n}}\subseteq \Lambda^1$ is a $\Q$-basis for $A$. Consider the  set $\set{\beta y_1,\ldots,\beta y_{4n}}$ which remains linearly independent over $\Q$ as $A$ is a division algebra, thus a $\Q$-basis for $A$. We have $\|\sigma(\beta y_i)\|^2 = \|(I,b_\alpha)\|$.  Indeed
    \begin{equation*}
        \begin{split}
            \|\sigma(\beta y_i)\|^2 &= Tr_{K/\Q}(tr_{A/K}(\alpha (\beta y_i)\overline{(\beta y_i)}))\\
            &=  Tr_{K/\Q}(tr_{A/K}(\alpha \beta y_i \overline{y_i}\overline{\beta}))\\
            &= Tr_{K/\Q}(tr_{A/K}(\alpha \beta (n_{A/K}(y_i)\overline{\beta}))\\
            &= Tr_{K/\Q}(tr_{A/K}(\alpha \beta\overline{\beta}))\\
            &= \|(I,b_\alpha)\|.
        \end{split}
    \end{equation*} This proves that $\set{\sigma(\beta y_i)| i=1,\ldots,4n}$ is an $\R$-basis for $\R^{4n}$ coming from minimum vectors.
\end{proof}

To conclude this subsection on characterizing well-rounded lattices, we summarize the different constructions of well-rounded lattices found in Table \ref{tab:cases}.

\begin{table}[h!]
\begin{tabular}{|l|l|l|c|}
\hline
$n$& $A$ & $\Lambda^1$ & $|\Lambda^1|$\\
\hline
&$\left(\frac{-1,-1}{\Q}\right)$  & binary tetrahedral group & 24 \\
1&$\left(\frac{-1,-1}{\Q}\right)$& binary dihedral group & 8 \\
&$\left(\frac{-1,-3}{\Q}\right)$& binary dihedral group & 12 \\\hline
&$\left(\frac{-1,-1}{\Q(\sqrt{2})}\right)$ & binary octahedral group & 48 \\
&$\left(\frac{-1,-1}{\Q(\sqrt{5})}\right)$ &  binary icosahedral group & 120 \\
2&$\left(\frac{-1,(\zeta_{2m}+\zeta_{2m}^{-1})^2-4}{\Q(\zeta_{2m}+\zeta_{2m}^{-1})}\right)$ & binary dihedral group &$4m=16,20,24$\\
\hline
$\geq3$&$\left(\frac{-1,(\zeta_{2m}+\zeta^{-1}_{2m})^2-4}{\Q(\zeta_{2m}+\zeta_{2m}^{-1})}\right)$ &binary dihedral group & $4m$ where $\varphi(2m)= 2n$\\
\hline
\end{tabular}   
\caption{\label{tab:cases}
A list of the possible totally definite quaternion $K$-algebras $A$ for which well-rounded lattices exist, described in terms of $A$ and $\Lambda^1$.}
\end{table}

\subsection{More properties}

The above proofs and computations allow us to explicit a basis of minimum vectors.

\begin{proposition}\label{Cor: Explicit basis}
    Suppose $(\Lambda,b_\alpha)$ is well-rounded and let $I$ be a right ideal of $\Lambda$.  Let $\beta \in I$ such that $\beta \in S(I,b_\alpha)$ (take $\beta = 1$ if $I = \Lambda$).
    \begin{itemize}
        \item If $\Lambda^1$ is isomorphic to one of the exceptional groups with generators $x,y,z$ whose relations are described in Lemma \ref{Lemma: classification of reduced norm 1 group}, then the set 
\[\set{\sigma(\beta z^ky^\ell)|k \in \set{0,\ldots,\varphi(c)-1}, \ell\in\set{0,1}}\]
    is an $\R$-basis of minimum vectors for $\R^{4n}$
     where $c$ is the order of the corresponding generator $z$ in $\Lambda^1$. 
     \item If $\Lambda^1$ is isomorphic to a binary dihedral group of order $4m$ with generators $x,y$ whose relations are described in Lemma $\ref{Lemma: classification of reduced norm 1 group}$, then the set 
     \[ \set{ \sigma(\beta y^k x^\ell) | k=0,\ldots, \varphi(2m)-1, \ell =0,1}\]is an $\R$-basis of minimum vectors for $\R^{4n}$.
    \end{itemize}
\end{proposition}

We will illustrate the above result in Example \ref{ex:zeta14}.

Let $\Lambda_1,\Lambda_2$ be $\I_K$-orders in $A$. The lattices $ \Lambda_1, \Lambda_2$ are said to be of the same type if there exists $u \in A^\times$ such that $\Lambda_1 = u^{-1}\Lambda_1 u$ \cite[Definition 17.4.1]{voight2021quaternion}, thus they are related by conjugation. Moreover, two orders are of the same type if they are isomorphic as $\I_K$-algebras \cite[Lemma 17.4.2]{voight2021quaternion}. The next proposition says that orders of the same type produce the same lattice.
\begin{proposition}\label{Prop: isomorphic orders are the same lattice}
    Let $\Lambda,\Lambda'$ be $\I_K$-orders of the same type and let $\alpha \in K^\times$ be  totally positive. Then $(\Lambda,b_\alpha) = (\Lambda',b_\alpha)$.
\end{proposition}
\begin{proof}
    Let $u \in A^\times$ such that $\Lambda' = u^{-1}\Lambda u$ and $\set{x_1,\ldots,x_{4n}}$ be a $\Z$-basis for $\Lambda$, then $\set{u^{-1}x_1 u,\ldots u^{-1}x_{4n}u}$ is a $\Z$-basis for $\Lambda'.$ It suffices to show that the Gram matrix of $(\Lambda,b_\alpha)$ and that of  $(\Lambda',b_\alpha)$ are the same. We will make use of the fact that for any $\boldsymbol{\lambda},\textbf{x},\textbf{y} \in A_\R$, then $b_\alpha(\boldsymbol{\lambda} \textbf{x},\textbf{y}) = b_\alpha( \textbf{x},\overline{\boldsymbol{\lambda}}\textbf{y})$ by Proposition \ref{Prop: positive definite bilinear form}. For any $1\leq i,j \leq 4n$, then 
    \begin{equation*}
        \begin{split}
            b_\alpha(u^{-1}x_iu, u^{-1}x_ju) &=b_\alpha(x_iu, \overline{u^{-1}}u^{-1}x_ju)\\
           &=  Tr_{K/\Q}(tr_{A/K}(\alpha (x_iu)\overline{(\overline{u^{-1}}u^{-1}x_ju)}))\\
            &= Tr_{K/\Q}(tr_{A/K}(\alpha (x_iu)\overline{n_{A/K}(u^{-1})x_j u}))\\
            &= Tr_{K/\Q}(tr_{A/K}(x_i u \overline{u}\overline{x_j}\overline{n_{A/K}(u^{-1}}))\\
            &= Tr_{K/\Q}(n_{A/K}(u)n_{A/K}(u^{-1})tr_{A/K}(x_i\overline{x_j}))\\
            &=Tr_{K/\Q}(tr_{A/K}(x_i \overline{x_j}))\\
            &= b_\alpha(x_i,x_j).
        \end{split}
    \end{equation*}
 \end{proof}

Let $B_1 = \left(\frac{-1,-1}{\Q}\right), B_2 = \left(\frac{-1,-3}{\Q}\right)$. We define the following three orders
\begin{equation*}
    \begin{split}
        & \Lambda_1 := \Z \oplus{} \Z i \oplus{}\Z j \oplus{} \Z ij \subset B_1,\\
        &\Lambda_2 := \Z \oplus{} \Z i \oplus{} \Z j \oplus{}\Z  \frac{1+i+j+ij}{2} \subset B_1,\\
        & \Lambda_3 := \Z \oplus{} \Z i \oplus{} \Z \frac{1+j}{2} \oplus{}\Z \frac{i+ij}{2} \subset B_2.
    \end{split}
\end{equation*} Fixing $\alpha \in \Q^\times_{>0}$, then $(\Lambda_1,b_\alpha)$ is the $\Z^4$ lattice, $(\Lambda_2,b_\alpha)$ is the $D_4$ lattice defined in Example \ref{Example: Hurwitz order over the rationals}, and $(\Lambda_3,b_\alpha)$ is the $A_2\oplus{}A_2$ lattice ($A_2$ denotes the hexagonal lattice), all scaled by a factor of $\alpha.$

We conclude this subsection with a classification theorem.

\begin{theorem}\label{Prop: wellrounded ideal lattice for the case over Q}
    Let $A$ be a totally definite quaternion algebra over $\Q$ and $\alpha \in \Q^\times_{>0}$. There are only three  well-rounded ideal lattices coming from a $\Z$-order which are the $\Z^4,D_4$ and $A_2\oplus{}A_2$ lattices.
\end{theorem}
\begin{proof}
    By Table \ref{tab:cases}, it suffices to look at $A = B_1$ or $A=B_2$ and to consider a $\Z$-order $\Lambda$ whose reduced norm one subgroup $\Lambda^1$ is isomorphic to the binary tetrahedral group of order 24, the binary dihedral group of order 8 in $A_1$ or the binary dihedral group of order 12 in $A_2.$ Since $A$ is a totally definite quaternion algebra over $\Q$, then $\Lambda^1 = \Lambda^\times$ by Lemma \cite[Lemma 11.5.9]{voight2021quaternion}. By \cite[Theorem 11.5.14]{voight2021quaternion}, $\Lambda$ is isomorphic to $\Lambda_1,\Lambda_2$ in $B_1$ or $\Lambda_3$ in $B_2$. So $\Lambda$ is a conjugate of $\Lambda_i$ for $i=1,2,3$ and by Proposition \ref{Prop: isomorphic orders are the same lattice}, $(\Lambda,b_\alpha) = (\Lambda_i,b_\alpha)$ for each $i=1,2,3.$ 
\end{proof}

\section{Examples and Discussions}
\label{sec:ex}
In this section, the computations are done using MAGMA \cite{MR1484478}. Standard notations for lattices (e.g., $D_4$, $A_2$, $A_6$) and the properties of these lattices are detailed in \cite{conway2013sphere}.

Theorem \ref{Prop: wellrounded ideal lattice for the case over Q} provides a classification of well-rounded lattices obtained from a $\Z$-order of a totally definite quaternion algebra over $\Q$, which exhibits three well-rounded lattices in dimension $4$: $\Z^4$, $D_4$ and $A_2\oplus A_2$.  
It is natural to wonder what happens when $K=\Q$ is replaced by $K$ a number field of degree $\geq 2$. 
We start with an example. 


\begin{example}
\label{ex:zeta14}
Let $\zeta_{14}$ be a $14$th primitive root of unity and set $\theta = \zeta_{14}+\zeta^{-1}_{14}$. 
Consider the totally definite quaternion algebra $A= \left(\frac{-7,-1}{K}\right)$ where $K = \Q(\theta)$ of degree 3, with $\I_K=\Z[\theta]$. Let $\set{1,i,j,ij}$ be the standard $K$-basis of $A$.  
Set $\beta=\theta^2-3\theta - 3$ and $\gamma = 7+i$. 
Let $\Lambda$ be the $\I_K$-order generated by the set 
\[
\left\{1,\frac{\gamma}{2},  \frac{\beta \gamma}{14}, -j, \frac{\gamma j}{2}, \frac{\beta\gamma j}{14}\right\}.
\]
Then $\Lambda^1$ is isomorphic to the binary dihedral group of order $28$ with generators 
$$x = -j, y=\frac{\theta^2- \theta -1}{2}+ \frac{\beta i}{14}.$$ Choose $\alpha = 1$, then $(\Lambda,b_1)$ is well-rounded by Theorem \ref{Thm: Characterization of WR lattice} and a $\Q$-basis for $A$ is $B= \set{y^i x^j| i=0,\ldots,5, j=0,1}$ by Proposition \ref{Cor: Explicit basis}. A $\Z$-basis for $\Lambda$ is 
    \[\left \{1, \theta,\theta^2, \frac{\gamma}{2},\frac{\theta\gamma}{2} ,\beta'\frac{\gamma
    }{14}, -j, -\theta j, -\theta^2 j, \frac{\gamma j}{2},\frac{\theta\gamma j}{2}, \beta'\frac{\gamma j}{14} \right\}\]
    with $\beta'=\theta^2+4\theta+4$.
    
    A Gram matrix for the ideal lattice $(\Lambda,b_1)$ is 
    \[ G = \setcounter{MaxMatrixCols}{12}\begin{pmatrix}
      6 & 2 & 10 & 21 & 7 & 21 & 0 & 0 & 0 & 0 & 0 & 0 \\
2 & 10 & 8 & 7 & 35 & 28 & 0 & 0 & 0 & 0 & 0 & 0 \\
10 & 8 & 26 & 35 & 28 & 49 & 0 & 0 & 0 & 0 & 0 & 0 \\
21 & 7 & 35 & 84 & 28 & 84 & 0 & 0 & 0 & 0 & 0 & 0 \\
7 & 35 & 28 & 28 & 140 & 112 & 0 & 0 & 0 & 0 & 0 & 0 \\
21 & 28 & 49 & 84 & 112 & 140 & 0 & 0 & 0 & 0 & 0 & 0 \\
0 & 0 & 0 & 0 & 0 & 0 & 6 & 2 & 10 & -21 & -7 & -21 \\
0 & 0 & 0 & 0 & 0 & 0 & 2 & 10 & 8 & -7 & -35 & -28 \\
0 & 0 & 0 & 0 & 0 & 0 & 10 & 8 & 26 & -35 & -28 & -49 \\
0 & 0 & 0 & 0 & 0 & 0 & -21 & -7 & -35 & 84 & 28 & 84 \\
0 & 0 & 0 & 0 & 0 & 0 & -7 & -35 & -28 & 28 & 140 & 112 \\
0 & 0 & 0 & 0 & 0 & 0 & -21 & -28 & -49 & 84 & 112 & 140
    \end{pmatrix}\] with disc$(\Lambda,b_\alpha)$ = $\det G = 282475249=7^{10}$.

Let $\Lambda' \subseteq \Lambda$ be the $\Z$-submodule spanned by $B$, then $(\Lambda',b_1)$ is a sublattice of $(\Lambda,b_1)$ with a Gram matrix  
\[ G' = \setcounter{MaxMatrixCols}{12}\begin{pmatrix}
    6&1&-1&1&-1&1&0&0&0&0&0&0\\
    1&6&1&-1&1&-1&0&0&0&0&0&0\\
    -1&1&6&1&-1&1&0&0&0&0&0&0\\
    1&-1&1&6&1&-1&0&0&0&0&0&0\\
    -1&1&-1&1&6&1&0&0&0&0&0&0\\
    1&-1&1&-1&1&6&0&0&0&0&0&0\\
    0&0&0&0&0&0&6&1&-1&1&-1&1\\
    0&0&0&0&0&0&1&6&1&-1&1&-1\\
    0&0&0&0&0&0&-1&1&6&1&-1&1\\
    0&0&0&0&0&0&1&-1&1&6&1&-1\\
    0&0&0&0&0&0&-1&1&-1&1&6&1\\
    0&0&0&0&0&0&1&-1&1&-1&1&6
\end{pmatrix}\] where $\det G' = 282475249=7^{10}$. This means that $(\Lambda',b_1)= (\Lambda,b_1)$ and thus $\Lambda = \Lambda'$.
 
The resulting lattice is the lattice $A_6^*\oplus A_6^*$ where $A_6^*$ denotes the dual lattice of $A_6$ and it is a known fact that the minimum vector in $A_6^*$ has norm $6$. Since $\det G'= M'M'^T \neq 0$, this implies that the generator matrix $M'$ with rows containing vectors $\sigma(y^i x^j)$ for $ i=0,\ldots,5, j =0,1$ has nonzero determinant and thus is linearly independent over $\R$. So $\sigma(B) = \set{\sigma(y^i x^j)|i=0,\ldots,5, j=0,1}$ is an $\R$-basis for $\R^{12}$ consisting of minimum vectors of $(\Lambda,b_1)$.

\end{example}

Classification of well-rounded lattices from totally definite quaternion algebras with $[K:\Q]\geq 2$ would be an interesting open question. In particular,  the choice of an $\I_K$-order plays a crucial role in constructing well-rounded ideal lattices arising from quaternion algebras, as shown in Theorem \ref{Thm: Characterization of WR lattice}. Several known orders have reduced norm one groups isomorphic to exceptional groups, as described in \cite[Proposition 3]{maire2006cancellation}, together with the necessary conditions in \cite[Proposition 5]{maire2006cancellation} for an order to admit such an isomorphism. Since isomorphic orders yield the same lattice by Proposition \ref{Prop: isomorphic orders are the same lattice}, a natural direction for further work is to construct new orders that are not isomorphic to these known examples yet whose reduced norm one subgroup is isomorphic to one of the groups listed in Table \ref{tab:cases}.

Corollary \ref{corr: wellrounded lattice from order is sufficient to ideal lattice to be well rounded}
provides a sufficient condition for an ideal lattice $(I,b_\alpha)$ to be well-rounded. 
Example \ref{ex: lattice from ideal} below shows that when considering a right ideal $I$ of an $\I_K$-order $\Lambda$, it is possible to get a well-rounded lattice $(I,b_\alpha)$ that is different from the well-rounded lattice $(\Lambda,b_\alpha)$ where $\alpha \in \Q^\times_{>0}$.

\begin{example}\label{ex: lattice from ideal}
    Let $A = \left( \frac{-1,-1}{\Q(\sqrt{3})}\right)$, $\alpha = \frac{1}{2}$ and consider the maximal order $\Lambda$ generated by 
   $\set{-1, i ,\frac{\sqrt{3}i+j}{2},\frac{\sqrt{3}+ij}{2}}$. Consider the following right ideal $I$ of $\Lambda$ generated by the set \[\set{\sqrt{3}-1, 2i,(\sqrt{3}+1)i, 1+\frac{\sqrt{3}i+j}{2}, \frac{\sqrt{3}+2i+ij}{2}}.\] Using MAGMA \cite{MR1484478} to compute the $\Z$-basis for $\Lambda$ and $I$, we obtain the following Gram matrix for $(\Lambda,b_\frac{1}{2})$: 
   \[G_\Lambda:=  {\begin{pmatrix}
       2&0&0&0&0&0&0&-3\\
       0&6&0&0&0&0&-3&0\\
       0&0&2&0&0&3&0&0\\
       0&0&0&6&3&0&0&0\\
       0&0&0&3&2&0&0&0\\
       0&0&3&0&0&6&0&0\\
       0&-3&0&0&0&0&2&0\\
       -3&0&0&0&0&0&0&6
   \end{pmatrix}},
   \]
   and the following Gram matrix for and $(I,b_\frac{1}{2})$: 
   \[
   G_I = {\begin{pmatrix}
       8&-12&0&0&-2&6&3&-3\\
       -12&24&0&0&6&-6&-3&9\\
       0&0&8&4&0&6&4&0\\
       0&0&4&8&3&3&2&6\\
       -2&6&0&3&4&0&0&6\\
       6&-6&6&3&0&12&6&0\\
       3&-3&4&2&0&6&4&0\\
       -3&9&0&6&6&0&0&12
   \end{pmatrix}}. \]
   The determinants are 81 and 1296 respectively.  It can be checked using MAGMA that $\Lambda^1$ is isomorphic to the binary dihedral group of order 24 and by Theorem \ref{Thm: Characterization of WR lattice}, $(\Lambda,b_\frac{1}{2})$ is well-rounded and that $\|(\Lambda,b_\frac{1}{2})\| =2$ by Proposition \ref{Prop: Lower bound for Hermite invariant for quaternion algebra}, this implies that $(I,b_\frac{1}{2})$ is well-rounded by Corollary \ref{corr: wellrounded lattice from order is sufficient to ideal lattice to be well rounded}. We will show that the two lattices are not similar. Recall that 
for $L_1,L_2$ two lattices with Gram matrices $G_1,G_2$ respectively, we say that $L_1$ and $L_2$ are similar if there exist $r >0$ and a unimodular matrix $U$ such that $G_1 = r UG_2U^T.$ 
   
If the two lattices are similar, there exists $r>0$ such that $\det rG_\Lambda = \det G_I \implies r= \sqrt{2}$. This implies that $\|(I,b_\frac{1}{2})\| = 2 r = 2\sqrt{2}$. However, $(I,b_\frac{1}{2})$ is an integral lattice, so $\|(I,b_\frac{1}{2})\|$ is a positive integer and thus cannot be $2\sqrt{2}.$ Thus, the two lattices are not similar.
\end{example}

The above example motivates the search for lattices coming from right ideals. In particular, 
the converse of Corollary \ref{corr: wellrounded lattice from order is sufficient to ideal lattice to be well rounded} may not be true. Fix $n= [K:\Q]$, one may search for right ideals $I$ of an $\I_K$-order $\Lambda$ for which $\Lambda^1$ is not among the groups in Table \ref{tab:cases}, yet the resulting ideal lattice remains well-rounded.

\newpage
\bibliographystyle{alpha}
\bibliography{references.bib}
\end{document}